\documentclass[11pt]{amsart}
  \RequirePackage{amsmath}
\RequirePackage{amssymb} \RequirePackage{amsthm}
\RequirePackage{epsfig} 
\usepackage{amsmath}
\usepackage{amssymb}
\usepackage{amsthm}

\newcommand{\C}{\mathbb{C}}
\newcommand{\D}{\mathbb{D}}
\newcommand{\N}{\mathbb{N}}
\renewcommand{\H}{\mathcal{H}}

\newcommand{\om}{\omega}
\renewcommand{\a}{\alpha}

\renewcommand{\d}{\delta}

\newcommand{\og}{\mathrm{O}}

\newtheorem{theorem}{Theorem}
\newtheorem{lemma}[theorem]{Lemma}

\theoremstyle{definition}

\theoremstyle{remark}

\begin{document}

\title[Jointly maximal products in  weighted growth spaces] {Jointly maximal products in  weighted growth spaces}

\author{Janne~Gr\"ohn}
\address{University of Eastern Finland, P.O.Box 111, 80101 Joensuu, Finland}
\curraddr{Universitat Aut\`onoma de Barcelona, Departament de Matem\`atiques, 08193 Bellaterra, Barcelona, Spain}
\email{janne.grohn@uef.fi}

\author{Jos\'e \'Angel Pel\'aez}

\address{Departamento de An아lisis Matem아tico, Universidad de M아laga, Campus de
Teatinos, 29071 M아laga, Spain} \email{japelaez@uma.es}

\author{Jouni R\"atty\"a}
\address{University of Eastern Finland, P.O.Box 111, 80101 Joensuu, Finland}
\email{jouni.rattya@uef.fi}

\thanks{This research was supported in part by the Ram\'on y Cajal program
of MICINN (Spain), Ministerio de Edu\-ca\-ci\'on y Ciencia, Spain,
(MTM2011-25502), from La Junta de Andaluc{\'i}a, (FQM210) and
(P09-FQM-4468), MICINN- Spain ref.
MTM2011-26538, European Networking Programme HCAA of the European
Science Foundation and the Finnish Cultural Foundation.}

\date{\today}

\keywords{Doubling function, infinite product, zero distribution}
\subjclass[2010]{Primary 30J99}

\begin{abstract}
It is shown that for any non-decreasing, continuous and unbounded doubling function $\om$ on $[0,1)$,
there exist two analytic infinite products $f_0$ and~$f_1$ such that the
asymptotic relation $|f_0(z)| + |f_1(z)| \asymp \om(|z|)$ is satisfied for all $z$ in the unit disc.
It is also shown that both functions $f_j$ for $j=0,1$ satisfy $T(r,f_j)\asymp\log\omega(r)$, as $r\to1^-$,
and hence give examples of analytic functions for which the Nevanlinna characteristic admits the regular
slow growth induced by $\omega$.
\end{abstract}
\maketitle
\section{Introduction and results}

Let $\H(\D)$ denote the algebra of all analytic functions in the unit disc $\D$ of the complex plane $\C$.
To consider the growth and the zero distribution of functions in $\H(\D)$, we use the following classical notation.
The \emph{non-integrated counting function} $n(r,f,0)$ counts
the zeros of $f$ in $\{ z\in\C : |z| \leq r \}$ according to multiplicities. Quantities $M_\infty(r,f)$, $M_p(r,f)$, where $0<p<\infty$,
$N(r,f,a)$, where $a\in\C$, and $T(r,f)$ denote the \emph{maximum modulus} of $f$, the \emph{$L^p$-mean} of $f$,
the \emph{integrated counting function of $a$-points} of $f$ and
the \emph{Nevanlinna characteristic} of $f$, respectively. We also employ the notation $a\asymp b$, which
is equivalent to the conditions $a\lesssim b$ and $b\lesssim a$,
where the former means that there exists a constant $C>0$ such that $a \leq C b$,
and the latter is defined analogously.

Let $\om : [0,1) \to (0,\infty)$ be non-decreasing, continuous and unbounded.
Such a function $\om$ is said to be \emph{doubling}, if there exists a constant $B>1$ such that
    \begin{equation} \label{eq:doubling}
      \om(1-r/2) \leq B \, \om (1-r), \quad 0<r\leq 1.
    \end{equation}
The following result shows that for
any doubling function $\om$ there exist two jointly maximal products
in the sense that the sum of their moduli behaves asymptotically as $\omega(|z|)$ in $\D$.


\begin{theorem} \label{Theorem:CanonicalProduct}
Let $\om: [0,1) \to (0,\infty)$ be doubling. Then, there exist
$f_0,f_1\in\H(\D)$ such that
    \begin{equation} \label{eq:star}
      |f_0(z)|+|f_1(z)|\asymp \om(|z|), \quad z \in\D,
    \end{equation}
where both functions $f_j$ for $j=0,1$ satisfy $n(r,f_j,0)=\og\left( (1-r)^{-1}\right)$, as $r\to 1^-$. Moreover,
\begin{equation}\label{n}
M_p(r,f_j)\asymp \om(r), \quad r\to 1^-,
\end{equation}
for all $0<p\leq \infty$, and
    \begin{equation}\label{eq:neweq}
       T(r,f_j)\asymp N(r,f_j,a)\asymp
      \log \, \om(r), \quad r\to 1^-,
    \end{equation}
    for all $a\in\C$.
\end{theorem}

The main advantage of our self-contained and constructive proof of Theorem~\ref{Theorem:CanonicalProduct} is that
the zero distribution of the products $f_0$ and $f_1$
is explicit. These products are similar to those applied
to study the zero distribution of functions in weighted Bergman spaces \cite[Section~$3$]{PelRat}.
Note also that our argument gives an alternative way to prove \cite[Theorem~3.15]{PelRat}, whose original proof
is based on certain lacunary series. In fact,  Theorem~\ref{Theorem:CanonicalProduct}
generalizes \cite[Theorem~3.15]{PelRat} to doubling functions.

The asymptotic relation \eqref{eq:star} reproduces  recent results  \cite[Theorem~$1.1$]{KP} and \cite[Lemma~1]{AD2012}
concerning regularly growing analytic functions in $\D$.
Proofs of \cite[Lemma~1]{AD2012} and \cite[Theorem~$1.1$]{KP} rest upon the use of lacunary series, which
have been the key tool to solve similar problems in the existing literature.
The pioneering result \cite[Proposition $5.4$]{RU}, which concerns \eqref{eq:star} for $\om(r) = 1/(1-r)$,
have been a source of inspiration for several authors. For example, \cite[Theorem~1.2]{GPPR} proves \eqref{eq:star}
for $\om(r) = -\log (1-r)$. It is well known that \eqref{eq:star} has many applications in the operator theory;
for details, we refer to~\cite{AD2012}.

Construction of a product whose Nevanlinna characteristic admits a pregiven asymptotic growth
has been studied by several authors. In particular, it is known that whenever $\Lambda(r)$
exceeds the growth of $-\log(1-r)$ as $r\to 1^-$, then there exists a product $f$
whose Nevanlinna characteristic behaves asymptotically as $\Lambda(r)$ \cite[Theorem~1]{S:1966}.
The asymptotic formula
\eqref{eq:neweq} shows that we can find an infinite analytic product in $\D$ such that its Nevanlinna
characteristic grows asymptotically as the logarithm of a pregiven doubling $\om$, and hence we can
prescribe characteristics growing slower than $-\log(1-r)$ as $r\to 1^-$.
%
The method and the construction in
the proof of Theorem~\ref{Theorem:CanonicalProduct} are different from those employed in \cite{S:1966}.

%
%



The current state of the development concerning complex linear differential equation $f'' + A(z) f = 0$  in $\D$
allows us to deduce a significant amount of information on solutions $f$,
whenever we can analyze the coefficient $A$ in detail.
If we take $A$ to be one of the functions $f_0$ and $f_1$ in the Theorem~\ref{Theorem:CanonicalProduct},
then we get an important and intriguing
family of examples of such differential equations. These particular equations are way too complicated to be solved explicitly,
but the growth and the oscillation of their solutions are well understood
due to the asymptotic properties satisfied by the coefficient $A$.
To be brief with regards to this matter, we settle to mention two cases in the recent literature of which the first one concerns
%
polynomial regular functions.
This class of regularly growing analytic functions in $\D$ arises naturally in the theory of
ODEs \cite{FR}. In the sense of linear differential equations,
polynomial regular functions play a similar role in the unit disc as polynomials do in the complex plane.
For a more general example, see \cite{Sepa}.


\section{Proof of Theorem~\ref{Theorem:CanonicalProduct}}

The proof of Theorem~\ref{Theorem:CanonicalProduct} is divided in several steps.
The point of departure is the construction of the infinite products $f_0,f_1\in\H(\D)$, which is followed
by a discussion of their growth. Finally, we consider the asserted asymptotic properties of the products $f_0$ and $f_1$.


\subsection{Construction of the products} \label{sec:aux}

Before going into the details of the construction, we note the following lemmas on doubling functions.


\begin{lemma} \label{lemma:doubling}
Let $\om: [0,1) \to (0,\infty)$ be doubling. If $B>1$ is the constant in~\eqref{eq:doubling},
then
    \begin{equation} \label{eq:dest}
      \om(t) \leq C \left( \frac{1-r}{1-t} \right)^\alpha \om(r), \quad 0\leq r \leq t < 1,
    \end{equation}
where $C = \max\big\{ B \, \om(1/2)/\om(0), B^2 \big\}$ and $\alpha=\log_2{B}$.
\end{lemma}

Conversely, it is obvious that, if $\om: [0,1) \to (0,\infty)$ is non-decreasing, continuous, unbounded
and it satisfies \eqref{eq:dest} for some $C>1$ and $\alpha>0$,
then $\om$ must be doubling.


\begin{proof}[Proof of Lemma~\ref{lemma:doubling}]
Since $\om$ is doubling, \eqref{eq:doubling} implies $\om(t) \leq B \, \om(2t - 1)$ for all $t\in [2^{-1}, 1)$.
Moreover, if $0\leq r\leq t<1$, then there exist unique constants $j,k\in\N \cup \{0\}$, $j\geq k$, such that
$t\in \big[ 1 - 2^{-j}, 1 - 2^{-j-1} \big)$ and $r\in \big[ 1- 2^{-k}, 1- 2^{-k-1} \big)$. If $k=0$, then
    \begin{equation*}
      \om(t)  \leq B^j \om\big( 2^j(t-1) +1 \big)
      \leq \frac{B^j \om(1/2)}{\om(0)} \om(r)
      \leq \frac{B \, \om(1/2)}{\om(0)} \left( \frac{1-r}{1-t} \right)^{\log_2{B}} \om(r),
    \end{equation*}
while if $k>0$, then
    \begin{equation*}
      \om(t)  \leq B^{j-k+1} \om\big( 2^{j-k+1}(t-1) +1 \big)
      \leq B^{j-k+1} \om(r)
      \leq B^2 \left( \frac{1-r}{1-t} \right)^{\log_2{B}} \om(r).
\end{equation*}
The assertion follows.
\end{proof}

The second lemma introduces a sequence of natural numbers depending
on the growth of the doubling function $\omega$. This sequence is the foundation
of our construction.


\begin{lemma} \label{lemma:2}
Let $\om: [0,1) \to (0,\infty)$ be doubling. Then, there exist a sequence $\{ n_k\}_{k=1}^\infty$ of natural numbers,
real constants $\lambda$ and $\mu$, and a constant $d\in (0,1)$ such that
the sequence $\{ a_k \}_{k=1}^\infty$, defined by
    \begin{equation*}
      a_k=\frac{\om\big( 1 - 1/n_{k+2} \big)}{\om\big( 1 - 1/n_{k} \big)}, \quad  k\in\N,
    \end{equation*}
satisfies
    \begin{equation}\label{wlacunaryDoubleandj9}
      1<\lambda\le a_k \le\mu<\infty,
      \quad
      \frac{\log a_{k+1}}{\log a_k}<d\, \frac{n_{k+1}}{n_k},
      \quad k\in\N.
    \end{equation}
\end{lemma}


\begin{proof}
Let $\alpha>0$ and $C>1$ be the constants ensured by Lemma~\ref{lemma:doubling}.
Now, let~$\gamma$ be a~sufficiently large real constant such that
    \begin{equation} \label{eq:req}
      2^{\gamma-\alpha} \, C^{-1} > 1,
      \quad
      \frac{2\gamma+\alpha+\log_2C}{2\gamma-\alpha-\log_2C}
      < \frac{1}{2^{1/\alpha}} \left( \frac{2^{\gamma/\alpha}}{C^{1/\alpha}} - 1 \right).
    \end{equation}
Take $t_1 = 1/2$, and
define the sequence $\{ t_k \}_{k=1}^\infty$ inductively by $\om(t_{k+1})/\om(t_k)=2^{\gamma}$ for $k\in\N$.
Let $n_k = {\rm floor}\big( (1-t_k )^{-1} \big)$, where ${\rm floor}(x) = \max\, \{ n\in\N : n \leq x\}$.
By means of Lemma~\ref{lemma:doubling}, and the estimates $2^{-1} \le n_k(1-t_k)\le 1$,
we may define $1 < \lambda< \mu < \infty$ by
    \begin{equation*}
      a_k \leq \frac{C \, \om( t_{k+2} )}{ \big( n_k (1-t_k) \big)^\alpha \om(t_k)}
      \leq 2^\alpha C  \, \frac{\om( t_{k+2} )}{\om(t_{k+1})} \, \frac{\om( t_{k+1} )}{\om(t_k)} = 2^{2 \gamma+ \alpha} C = \mu,
    \end{equation*}
and
    \begin{equation*}
      a_k \geq \frac{\big( n_{k+2} (1-t_{k+2}) \big)^\alpha \om( t_{k+2} )}{ C \, \om(t_k)}
      \geq \frac{1}{2^\alpha C}\,  \frac{ \om( t_{k+2} )}{ \, \om(t_{k+1})} \,  \frac{ \om( t_{k+1} )}{ \, \om(t_k)} = 2^{2\gamma-\alpha} C^{-1}=\lambda,
    \end{equation*}
since these inequalities hold for all $k\in\N$. By Lemma~\ref{lemma:doubling}  we conclude
    \begin{equation} \label{eq:myestold}
      \begin{split}
        \frac{n_{k+1}}{n_k}
        & > \left( \frac{1}{1-t_{k+1}} - 1 \right) (1-t_k)
        > \frac{1-t_k}{1-t_{k+1}} - 1 \\
        & \geq \frac{1}{C^{1/\alpha}} \left( \frac{\om(t_{k+1})}{\om(t_k)} \right)^{1/\alpha} -1
        = \frac{2^{\gamma/\alpha}}{C^{1/\a}}-1, \quad k\in\N,
      \end{split}
    \end{equation}
and further by \eqref{eq:req}, we have
    \begin{equation*}
      \begin{split}
        \frac{\log a_{k+1}}{\log a_k} & \leq \frac{\log \mu}{\log \lambda}
        = \frac{2\gamma+\alpha+\log_2C}{2\gamma-\alpha-\log_2C}
        <\frac{1}{2^{1/\alpha}} \left( \frac{2^{\gamma/\alpha}}{C^{1/\alpha}} - 1 \right)
        <\frac{1}{2^{1/\alpha}} \, \frac{n_{k+1}}{n_k}.
      \end{split}
    \end{equation*}
This confirms the last inequality in \eqref{wlacunaryDoubleandj9} for $d=2^{-1/\alpha}$.
\end{proof}

Let $\{n_k\}_{k=1}^\infty$ be the sequence ensured by Lemma~\ref{lemma:2}, and  define
    \begin{equation*}
      f_j(z)=\prod_{k=1}^\infty \frac{ 1+a_{2k+j} z^{n_{2k+j}}}{1+a^{-1}_{2k+j} z^{n_{2k+j}}},\quad z\in\D, \quad j=0,1.
    \end{equation*}
Evidently both functions $f_j$ belong to $\H(\D)$, since all factors are bounded functions in~$\D$,
and according to~\eqref{wlacunaryDoubleandj9} the sum
    \begin{equation*}
      \sum_{k=1}^\infty \left| \frac{ 1+a_{2k+j} z^{n_{2k+j}}}{1+a^{-1}_{2k+j} z^{n_{2k+j}}} - 1 \right|
      \leq \sum_{k=1}^\infty \frac{a_{2k+j} - a_{2k+j}^{-1}}{1-a_{2k+j}^{-1}} \, |z|^{n_{2k+j}}
      \leq (1+\mu) \sum_{k=1}^\infty |z|^{n_{2k+j}}
    \end{equation*}
converges uniformly on compact subsets of $\D$.


\subsection{Growth estimates for the maximum modulus of the products}

To estimate the growth of $f_j$ for $j=0,1$ we define $r_{2m+j} = e^{-1/n_{2m+j}}$ for $m\in\N$, and write
   \begin{equation}\label{zqp4}
     |f_j(z)|=\left| \prod_{k=1}^m a_{2k+j} \frac{a_{2k+j}^{-1}+z^{n_{2k+j}}}{1+a^{-1}_{2k+j} z^{n_{2k+j}}}\right|
     \left|\prod_{k=1}^\infty \frac{ 1+a_{2(m+k)+j} z^{n_{2(m+k)+j}}}{1+a^{-1}_{2(m+k)+j}
         z^{n_{2(m+k)+j}}}\right|.
   \end{equation}
First, we prove that the infinite subproduct in \eqref{zqp4} is bounded in~$\D$. To this end,
let $\tau = 2^{\gamma/\alpha} C^{-1/\a}-1$ be the lower bound in \eqref{eq:myestold}.
According to \eqref{eq:req} we know that $\tau>1$, and
    \begin{equation} \label{eq:recur}
      \frac{n_{2(m+k)+j}}{n_{2m+j}} = \frac{n_{2(m+k)+j}}{n_{2(m+k)+j-1}} 
      \dotsb \frac{n_{2m+j+1}}{n_{2m+j}} \geq \tau^{2k},
      \quad k,m\in\N.
    \end{equation}
Since $h_1(x)= (y+x)/(1+y x)$ is increasing on $[0,1)$ for each $y\in[0,1)$, we obtain
    \begin{equation}\label{67}
      \begin{split}
        \left|\frac{ 1+a_{2(m+k)+j}  z^{n_{2(m+k)+j}}}{1+a^{-1}_{2(m+k)+j}  z^{n_{2(m+k)+j}}}\right|
        &=a_{2(m+k)+j}\left|\frac{ a^{-1}_{2(m+k)+j}+ z^{n_{2(m+k)+j}}}{1+a^{-1}_{2(m+k)+j} z^{n_{2(m+k)+j}}}\right|\\
        & \le a_{2(m+k)+j} \, \frac{ a^{-1}_{2(m+k)+j}+ |z|^{n_{2(m+k)+j}}}{1+a^{-1}_{2(m+k)+j} |z|^{n_{2(m+k)+j}}}\\
        &< \frac{1+a_{2(m+k)+j}\left(\frac1e\right)^{\frac{n_{2(m+k)+j}}{n_{2m+j}}}}{1+a^{-1}_{2(m+k)+j}
          \left(\frac1e\right)^{\frac{n_{2(m+k)+j}}{n_{2m+j}}}}
      \end{split}
    \end{equation}
for $|z|< r_{2m+j}$ and $k,m\in\N$.
Moreover, since $h_2(x,y)=(1+xy)/(1+x^{-1}y)$ is increasing
in both variables, provided that $x>1$ and $0\leq y<1$, estimates~\eqref{wlacunaryDoubleandj9},~\eqref{eq:recur}
and \eqref{67} imply
    \begin{equation}\label{zqp5}
      \left| \prod_{k=1}^\infty \frac{ 1+a_{2(m+k)+j} z^{n_{2(m+k)+j}}}{1+a^{-1}_{2(m+k)+j}
          z^{n_{2(m+k)+j}}}\right|
      <\prod_{k=1}^\infty\frac{1+\mu
        \left(\frac1e\right)^{\tau^{2k}}}{1+\mu^{-1}
        \left(\frac1e\right)^{\tau^{2k}}}
      \leq C^\star< \infty,
    \end{equation}
for $|z|< r_{2m+j}$ and $m\in\N$, where $C^\star>0$ is a constant independent of $m\in\N$.
Second, we proceed to derive an upper estimate for the maximum modulus of~$f_j$.
By means of \eqref{wlacunaryDoubleandj9}, \eqref{zqp4},
\eqref{zqp5} and the inequality $1-x \leq e^{-x}$ for $x\ge0$, we get
    \begin{align}
      |f_j(z)|&< C^\star \prod_{k=1}^m a_{2k+j}
      = C^\star \, \frac{\om\big(1-1/n_{2(m+1)+j}\big)}{\om\big(1-1/n_{2+j}\big)}
      \le C^\star \mu \, \frac{\om\big(1-1/n_{2m+j}\big)}{\om\big(1-1/n_{2+j}\big)} \notag\\
      &\le C^\star \mu \, \frac{\om(r_{2m+j})}{\om\big(1-1/n_{2+j}\big)},
      \quad |z|< r_{2m+j},\quad m\in\N. \label{zqp6}
    \end{align}
If $|z|\ge r_{2+j}$, then $r_{2(m-1)+j}\le|z|<r_{2m+j}$ for some $m\in\N\setminus\{1\}$.
Note that by \eqref{eq:recur} there exists $t\in\N$  such that $n_{2(m+t)+j} > 2 \,n_{2m+j}$ for all $m\in\N$.
Since $e^{-x}\le 1-x/2$ for $0\leq x\leq 1$, we conclude
    \begin{equation*}
      r_{2m+j} \leq 1 - (2 \, n_{2m+j})^{-1} < 1 - 1/n_{2(m+t)+j}, \quad m\in\N.
    \end{equation*}
Then \eqref{wlacunaryDoubleandj9}, \eqref{zqp6} and the inequality $1-x\le
e^{-x}$ for $0\leq x\leq 1$, give
    \begin{equation*}
      \begin{split}
        |f_j(z)|&<
        C^\star \mu \, \frac{\om\big( 1 - 1/n_{2(m+t)+j} \big)}{\om\big(1-1/n_{2+j}\big)}
        \leq C^\star \mu^{2+t} \, \frac{\om\big( 1 - 1/n_{2(m-1)+j} \big)}{\om\big(1-1/n_{2+j}\big)}\\
        & \leq C^\star \mu^{2+t} \, \frac{\om\big( r_{2(m-1)+j} \big)}{\om\big(1-1/n_{2+j}\big)}
        \leq C^\star \mu^{2+t}\, \frac{\om\big( |z| \big)}{\om\big(1-1/n_{2+j}\big)}.
      \end{split}
    \end{equation*}
Consequently, the maximum modulus of $f_j$ satisfies
    \begin{equation} \label{eq:maxmod}
      M_\infty(r,f_j) = \max_{|z|=r} \big| f_j(z) \big| \lesssim \om(r),\quad 0\leq r<1.
    \end{equation}


\subsection{Growth estimates for the minimum modulus of the products}

The following discussion shows that the difference between the maximum modulus and
the minimum modulus of $f_j$ for $j=0,1$ is small in a large subset of the unit disc.
Define $E_j = \bigcup_{m=1}^\infty I_{2m+j}$, where $I_{2m+j}$ is the closed interval whose endpoints are
    \begin{equation*}
      \min I_{2m+j} = \left( a_{2m+j}^{-n_{2m+j}^{-1}} \right)^{1-\d} \left( a_{2(m+1)+j}^{-n_{2(m+1)+j}^{-1}} \right)^\d, \quad m\in\N,
    \end{equation*}
and
    \begin{equation*}
      \max I_{2m+j}   = \left( a_{2m+j}^{-n_{2m+j}^{-1}} \right)^{\delta \frac{n_{2m+j}}{n_{2m+1+j}}}
      \left( a_{2(m+1)+j}^{-n_{2(m+1)+j}^{-1}} \right)^{1-\d \frac{n_{2m+j}}{n_{2m+1+j}} }, \quad m\in\N.
    \end{equation*}
Here $0<\d<1$ is a sufficiently small constant, which is to be determined later. According to
\eqref{wlacunaryDoubleandj9} all elements in the sequence $\big\{ a_m^{-1/n_m} \big\}_{m=1}^\infty$ belong
to the interval $(0,1)$, this sequence is strictly increasing, and it converges to $1$, as $m\to\infty$.
Moreover, $I_{2m+j} \subset  \big( a_{2m+j}^{-1/n_{2m+j}},  a_{2(m+1)+j}^{-1/n_{2(m+1)+j}} \big)$ for all $m\in\N$.
First, we prove that the infinite subproduct in \eqref{zqp4} is uniformly bounded away from zero
for $|z|\in E_j$. If $|z|\in I_{2m+j}$, then $|z|^{n_{2(m+k)+j}}<a_{2(m+k)+j}^{-1}$ for all $k\in\N$, and therefore
    \begin{equation} \label{eq:lest}
      \begin{split}
        \left|\frac{ 1+a_{2(m+k)+j}  z^{n_{2(m+k)+j}}}{1+a^{-1}_{2(m+k)+j}  z^{n_{2(m+k)+j}}}\right|
        & \ge a_{2(m+k)+j} \, \frac{ a^{-1}_{2(m+k)+j}- |z|^{n_{2(m+k)+j}}}{1-a^{-1}_{2(m+k)+j} \, |z|^{n_{2(m+k)+j}}}\\
        &=\frac{1-a_{2(m+k)+j}\, |z|^{n_{2(m+k)+j}}}{1-a^{-1}_{2(m+k)+j} \, |z|^{n_{2(m+k)+j}}}
      \end{split}
    \end{equation}
for $|z|\in I_{2m+j}$ and $k,\,m\in\N$. Since $h_3(x,y)=(1-x y)/(1-x^{-1}y)$ is decreasing
in both variables, when $x>1$ and $0\leq y<1$,
estimates \eqref{wlacunaryDoubleandj9}, \eqref{eq:recur} and \eqref{eq:lest}
imply that there exists a constant $C^\ast>0$, independent of $m\in\N$, such that
    \begin{align}
      & \left|\prod_{k=1}^\infty\frac{ 1+a_{2(m+k)+j} z^{n_{2(m+k)+j}}}{1+a^{-1}_{2(m+k)+j}
          z^{n_{2(m+k)+j}}}\right|\notag\\
      & \qquad \ge\prod_{k=1}^\infty\frac{1-a_{2(m+k)+j} \Bigg(a_{2m+j}^{-\d}
        \, a_{2(m+1)+j}^{-\frac{n_{2m+1+j}}{n_{2(m+1)+j}}\left( 1-\d \frac{n_{2m+j}}{n_{2m+1+j}}  \right)}\Bigg)^\frac{n_{2(m+k)+j}}{n_{2m+1+j}}}
      {1-a_{2(m+k)+j}^{-1}
        \Bigg(a_{2m+j}^{-\d}
        \, a_{2(m+1)+j}^{-\frac{n_{2m+1+j}}{n_{2(m+1)+j}}\left( 1-\d \frac{n_{2m+j}}{n_{2m+1+j}}  \right)}\Bigg)^\frac{n_{2(m+k)+j}}{n_{2m+1+j}}} \label{eq:llest}\\
      &  \qquad \geq \prod_{k=1}^\infty \frac{1-\mu \big( \lambda^{-\delta}\big)^{\tau^{2k-1}}}{1-\mu^{-1} \big( \lambda^{-\delta}\big)^{\tau^{2k-1}}}
      \geq C^\ast, \quad |z|\in I_{2m+j}, \quad m\in\N. \notag
    \end{align}
Second, we proceed to estimate the minimum modulus of $f_j$ on $E_j$. Note that
the last inequality in \eqref{wlacunaryDoubleandj9} implies
    \begin{equation} \label{eq:difiest}
      \begin{split}
        a_{2k+j} \, |z|^{n_{2k+j}} &\ge
        a_{2k+j} \left(a_{2m+j}^{-\frac1{n_{2m+j}}(1-\d)} a_{2(m+1)+j}^{-\frac1{n_{2(m+1)+j}}\d}\right)^{n_{2k+j}}\\
        & =\frac{a_{2k+j}}{a_{2m+j}^{\frac{n_{2k+j}}{n_{2m+j}}(1-\d)} a_{2(m+1)+j}^{\frac{n_{2k+j}}{n_{2(m+1)+j}}\d}}
        \ge\frac{a_{2k+j}}{a_{2k+j}^{d^{2(m-k)}(1-\d)} a_{2k+j}^{d^{2(m+1-k)}\d}}\\
        & \ge\frac{a_{2k+j}}{a_{2k+j}^{1-\d} \, a_{2k+j}^{d\d}}
        =a_{2k+j}^{\d(1-d)}
        \ge\lambda^{\d(1-d)} > 1
      \end{split}
    \end{equation}
for $|z|\in I_{2m+j}$ when $1\leq k \leq m$, and $m\in\N$; in particular, $|z|^{n_{2k+j}} > a_{2k+j}^{-1}$.
Moreover, choose $t\in\N$ sufficiently large such that $1-\lambda^{-1} \geq \tau^{-2t}$. Since
$h_4(x) = 1- (1-a)x^{-1} - a^{1/x} \geq 0$ for all $x\in [1,\infty)$, provided that $a\in(0,1)$, by applying
\eqref{wlacunaryDoubleandj9}  and \eqref{eq:recur}, we obtain
    \begin{equation} \label{eq:zt}
      \begin{split}
        |z| & \leq a_{2(m+1)+j}^{-1/n_{2(m+1)+j}} \leq \lambda^{-1/n_{2(m+1)+j}}
        \leq 1 - \left( 1 - \lambda^{-1} \right) n_{2(m+1)+j}^{-1} \\
        & \leq 1 - \tau^{-2t} \, n_{2(m+1)+j}^{-1} \leq 1 - 1/n_{2(m+1+t)+j}, \quad |z|\in I_{2m+j}, \quad m\in\N.
      \end{split}
    \end{equation}
Therefore \eqref{zqp4}, \eqref{eq:llest} and \eqref{eq:zt} yield
    \begin{equation} \label{eq:product}
      \begin{split}
        |f_j(z)|
        &\geq C^\ast \prod_{k=1}^m a_{2k+j} \left|\frac{a_{2k+j}^{-1}+z^{n_{2k+j}}}{1+a^{-1}_{2k+j} \, z^{n_{2k+j}}}\right|\\
        & = C^\ast \frac{\om \big( 1- 1/n_{2(m+1)+j} \big)}{\om\big( 1 - 1/n_{2+j} \big)}
        \prod_{k=1}^m\left|\frac{a_{2k+j}^{-1}+z^{n_{2k+j}}}{1+a^{-1}_{2k+j}  \, z^{n_{2k+j}}}\right| \\
        &\geq C^\ast \frac{\om \big( 1- 1/n_{2(m+1+t)+j} \big)}{\mu^t \, \om\big( 1 - 1/n_{2+j} \big)}
        \prod_{k=1}^m\frac{|z|^{n_{2k+j}}-a_{2k+j}^{-1}}{1-a^{-1}_{2k+j} |z|^{n_{2k+j}}} \\
        &\geq C^\ast \frac{\om( |z| )}{\mu^t \, \om\big( 1 - 1/n_{2+j} \big)}
        \prod_{k=1}^m\frac{|z|^{n_{2k+j}}-a_{2k+j}^{-1}}{1-a^{-1}_{2k+j} |z|^{n_{2k+j}}}
      \end{split}
    \end{equation}
for $|z|\in I_{2m+j}$ and $m\in\N$.
For our purposes, it suffices to show that the product in the last line of \eqref{eq:product} is uniformly
bounded away from zero for $|z|\in E_j$. To simplify computations, we prove that the reciprocal of this product
is uniformly bounded for such values of~$z$. Now, since $\log x\le x-1$ for $x\ge1$, we have
    \begin{equation*}
      \begin{split}
        \prod_{k=1}^m\frac{1-a^{-1}_{2k+j} |z|^{n_{2k+j}}}{|z|^{n_{2k+j}}-a_{2k+j}^{-1}}
        & =\exp\left(\sum_{k=1}^m\log\frac{1-a^{-1}_{2k+j} |z|^{n_{2k+j}}}{|z|^{n_{2k+j}}-a_{2k+j}^{-1}}\right)\\
        & \le\exp\left((1+\mu)\sum_{k=1}^m\frac{1-|z|^{n_{2k+j}}}{a_{2k+j} |z|^{n_{2k+j}}-1}\right)
      \end{split}
    \end{equation*}
for $|z|\in I_{2m+j}$ and $m\in\N$. By means of \eqref{eq:recur}, \eqref{eq:difiest}, and
the estimate $1-e^{-x} \leq x$ for $x\geq 0$, we conclude
    \begin{equation*}
      \begin{split}
        \prod_{k=1}^m\frac{1-a^{-1}_{2k+j} |z|^{n_{2k+j}}}{|z|^{n_{2k+j}}-a_{2k+j}^{-1}}
        &\le\exp\left(\frac{1+\mu}{\lambda^{\d(1-d)}-1}
          \, \sum_{k=1}^m(1-|z|^{n_{2k+j}})\right)\\
        &\le\exp\left(\frac{1+\mu}{\lambda^{\d(1-d)}-1}
          \, \sum_{k=1}^m\left(1-a_{2m+j}^{-n_{2k+j}/n_{2m+j}}\right)\right)\\
        &\le\exp\left(\frac{1+\mu}{\lambda^{\d(1-d)}-1}\, \log\mu
          \, \sum_{k=1}^m\frac{n_{2k+j}}{n_{2m+j}}\right)\\
        &\le\exp\left(\frac{1+\mu}{\lambda^{\d(1-d)}-1}\, \log\mu
          \, \sum_{k=0}^\infty\frac1{\tau^k}\right)
      \end{split}
    \end{equation*}
for $|z|\in I_{2m+j}$ and $m\in\N$, which gives the desired uniform lower bound for the product
in the last line of \eqref{eq:product}. Hence, by \eqref{eq:maxmod} and \eqref{eq:product}, we get
    \begin{equation} \label{eq:finalest}
      |f_j(z)|\asymp \om(|z|), \quad |z| \in E_j = \bigcup_{m=1}^\infty I_{2m+j}.
    \end{equation}


\subsection{The covering property of the sets where the products are maximal}

It remains to prove that the sets $E_0$ and $E_1$ induce a covering of $[\min I_2,1)$.
Note that the closed intervals $\{ I_{2m} \}_{m=1}^\infty$ are pairwise disjoint,
which is also true for $\{ I_{2m+1} \}_{m=1}^\infty$. Consequently, it is sufficient to show that
    \begin{equation} \label{eq:minmax}
      \min I_{2m+1} \leq \max I_{2m}, \quad \min I_{2(m+1)} \leq \max I_{2m+1}, \quad m\in\N.
    \end{equation}
We proceed to prove the first inequality in \eqref{eq:minmax}. By the definition of $I_{2m+j}$,
the first inequality in \eqref{eq:minmax} is equivalent to
    \begin{equation} \label{eq:aa}
      a_{2m+1}^{-\frac{1-\delta}{n_{2m+1}}} a_{2m+3}^{-\frac{\delta}{n_{2m+3}}}
      \leq a_{2m}^{-\frac{\delta}{n_{2m+1}}} a_{2m+2}^{- \frac{1}{n_{2m+2}} \left( 1- \delta \frac{n_{2m}}{n_{2m+1}} \right)},
      \quad m\in\N.
    \end{equation}
By taking the logarithm to the base $a_{2m}$ on the both sides of \eqref{eq:aa}, and
then solving the resulting inequality with respect to $\delta$, we
conclude that the first inequality in \eqref{eq:aa} is valid if and only if
$\delta \leq T(m)$ for all $m\in\N$, where
    \begin{equation*}
      T(m) = \frac{\log_{a_{2m}} a_{2m+2}^{-\frac{1}{n_{2m+2}}} - \log_{a_{2m}} a_{2m+1}^{-\frac{1}{n_{2m+1}}}}
      {\frac{1}{n_{2m+1}} - \log_{a_{2m}} a_{2m+1}^{- \frac{1}{n_{2m+1}}} + \log_{a_{2m}} a_{2m+3}^{- \frac{1}{n_{2m+3}}}
        + \log_{a_{2m}} a_{2m+2}^{- \frac{n_{2m}}{n_{2m+1}n_{2m+2}}}}.
    \end{equation*}
Note that the denominator of $T(m)$ can be written in the form
    \begin{equation*}
      \begin{split}
        & \log_{a_{2m}} a_{2m+3}^{- \frac{1}{n_{2m+3}}}  - \log_{a_{2m}} a_{2m+1}^{- \frac{1}{n_{2m+1}}} \\
        & \qquad + \frac{n_{2m}}{n_{2m+1}} \left( \log_{a_{2m}} a_{2m+2}^{- \frac{1}{n_{2m+2}}} -  \log_{a_{2m}} a_{2m}^{-\frac{1}{n_{2m}}} \right) > 0,
        \quad m\in\N,
      \end{split}
    \end{equation*}
and hence $T(m)$ is strictly positive for all $m\in\N$. By means of \eqref{wlacunaryDoubleandj9} we get
    \begin{equation*}
      T(m) \geq \frac{(d-1) \log_{a_{2m}} a_{2m+1}^{-\frac{1}{n_{2m+1}}}}{\frac{1}{n_{2m+1}} - \log_{a_{2m}} a_{2m+1}^{- \frac{1}{n_{2m+1}}}}
      = \frac{(1-d) \log_{a_{2m}} a_{2m+1}}{1+\log_{a_{2m}} a_{2m+1} }, \quad m\in\N.
    \end{equation*}
This implies that, if
    \begin{equation} \label{eq:delta}
      0<\delta \le \frac{(1-d) \log_\mu \lambda}{1+\log_\lambda \mu},
    \end{equation}
then the first inequality in \eqref{eq:minmax} is satisfied for all $m\in\N$. The second inequality in \eqref{eq:aa}
follows by a similar argument, and the choice \eqref{eq:delta} for $\delta$ is again adequate. We conclude that
\begin{equation} \label{eq:nn}
|f_0(z)| + |f_1(z)| \asymp \om(|z|)
\end{equation}
for $ |z|\geq \min I_2$. Finally, since
$\big\{a_m^{-1/n_m}\big\}_{m=1}^\infty$ is strictly increasing, \eqref{eq:nn} holds also for $|z|\leq \min I_2$,
and hence $f_0$ and $f_1$ are analytic functions satisfying \eqref{eq:star}.


\subsection{Asymptotic properties of the products}

Product $f_j$ for $j=0,1$ has exactly $n_{2m+j}$ simple zeros on the each circle
$\big\{z:|z|=s_{2m+j}\big\}$, where $s_{2m+j}=a_{2m+j}^{-1/n_{2m+j}}$ for $m\in\N$.  Therefore, we obtain
    \begin{equation*}
      \begin{split}
        n_{2m+j}& \le
        n(s_{2m+j},f_j,0)=\sum_{k=1}^{m}n_{2k+j}=n_{2m+j}\sum_{k=1}^{m}\frac{1}{\frac{n_{2m+j}}{n_{2k+j}}}\\
        & \le n_{2m+j}\sum_{k=0}^\infty\frac1{\tau^k}\lesssim n_{2m+j}, \quad m\in\N,
      \end{split}
    \end{equation*}
by \eqref{eq:recur}. By applying the estimates $1-x< \log x^{-1} < 2(1-x)$, which are valid for $4^{-1} < x < 1$,
it follows that $n(s_{2m+j},f_j,0) \asymp (1-s_{2m+j})^{-1}$ for all $m\in\N$.
 Consequently,
$$n(r,f_j,0)= \og\left( (1-r)^{-1}\right),\quad r\to 1^-.$$

Now we observe that $E_0 \cup E_1 = [\min I_{2} ,1)$, so it is not possible that $\underline{d}(E_0) = \underline{d}(E_1) = 0$,
where
    \begin{equation*}
      \underline{d}(F) = \liminf_{r\to 1^-} \, \frac{m\big(F \cap [r,1)\big)}{1-r}
    \end{equation*}
is the \emph{lower density} of the set $F\subset [0,1)$, and where $m$ denotes the Lebesgue measure.
Consequently, for some $j=0,1$  we have  $\underline{d}(E_j)>0$, which together with the nature of the sets $E_j$, implies that
$\underline{d}(E_j)>0$ for both $j=0,1$.
Consequently,
  \eqref{eq:finalest} holds outside
a set $E_j^\star = [0,1) \setminus E_j$,  which satisfies
    \begin{equation*}
      \overline{d} ( E_j^\star ) = \limsup_{r\to 1^-} \, \frac{m\big(E_j^\star \cap [r,1)\big)}{1-r}
      <1, \quad j=0,1.
    \end{equation*}
   So, \eqref{eq:finalest}, \cite[Lemma~4.3]{JH}  and Lemma~\ref{lemma:doubling} yield
$M_p(r,f_j) \asymp w(r)$, as $r\to 1^-$, where the constants in the asymptotic relation are independent of
$0<p\le \infty$. This proves \eqref{n}.
 On the other hand,  for any $a\in\C$, Jensen's formula
and \eqref{eq:finalest} imply that
$N(r,f_j,a)\asymp \log \, \om(r)$ for $r \in [0,1) \setminus E_0^\star$.
The fact that the same estimate holds also without the exceptional set $E_0^\star$ follows again from~\cite[Lemma~4.3]{JH}
and Lemma~\ref{lemma:doubling}. Furthermore,
 $ \log \, \om(r)\asymp N(r,f_j,0) \lesssim T(r,f_j) \leq \log M_\infty(r,f_j)\asymp \log \, \om(r)$, as $r\to 1^-$,
again by Jensen's formula. This completes the proof of Theorem~\ref{Theorem:CanonicalProduct}.


\end{document}